\documentclass[11pt, leqno, twoside]{amsart}
\usepackage{amssymb}
\usepackage{amsmath, amsthm, amsfonts}

\def\leq{\leqslant}

\newtheorem{theorem}{Theorem}[section]
\newtheorem{lemma}{Lemma}[section]

\begin{document}
\numberwithin{equation}{section}

\def\C{ \mathbb{C}}
\def\Q{ \mathbb{Q}}
\def\R{ \mathbb{R}}
\def\P{ \mathbb{P}}
\def\N{ \mathbb{N}}
\def\Z{ \mathbb{Z}}

\def\ord{{\rm ord} }
\def\mult{{\rm mult}}
\def\endProof {\rule[-0.5mm]{1.2ex}{1.2ex}}
\author{Nguyen Van Chau}
\address{Institute of Mathematics\newline\indent
18 Hoang Quoc Viet
\newline\indent
10307 Hanoi, Vietnam}
\email{nvchau@math.ac.vn}
\title[singularity and  non-proper value set of polynomial maps of $\C^2$]{A note on singularity and  non-proper value set of polynomial maps of $\C^2$}
\date{April 25, 2007.}
\thanks{Supported in part by the National Basic Program on Natural Science, Vietnam, and ICTP, trieste, Italy.}
\maketitle

\begin{abstract}
Some properties of the relation between the singular point set and
the non-proper value curve of polynomial maps of $\C^2$ are
expressed in terms of Newton-Puiseux expansions.
\end{abstract}

\section{Introduction}

Recall that the so-called {\it non-proper value set } $A_f$ of a
polynomial map $f=(P,Q):\C^2\longrightarrow \C^2$, $P,Q\in
\C[x,y]$,   is the set of all point $b\in \C^2$ such that there
exists a sequence $\C^2 \ni a_i\rightarrow \infty $ with
$f(a_i)\rightarrow b$. The set $A_f$ is empty if and only if $f$
is proper and $f$ has a  polynomial inverse if and only if $f$ has
not singularity and $A_f=\emptyset$. The mysterious Jacobian
conjecture (JC) (See [4] and [8]), posed first by Keller in 1939
and still open, asserts that if $f$ has not singularity, then $f$
has a polynomial inverse.
 In other words, (JC) shows  that the non-proper value set of a non-singular polynomial map of $\C^2$ must be empty.
 In any way one may think that the knowledge on the relation between the
singularity set and the non-proper value set  should be useful in
pursuit of this conjecture.

Jelonek in [9] observed  that for non-constant polynomial map $f$
of $\C^2$ the non-proper value set $A_f$, if non empty, must be a
plane curve such that each of its irreducible components can be
parameterized by a non-constant polynomial map from $\C$ into
$\C^2$.  Following [6], the non-proper value set  $A_f$ can be
described in term of Newton-Puiseux expansion as follows. Denote
by $\Pi$ the set of all  finite fractional power series
$\varphi(x,\xi)$ of the form
\begin{equation}
\varphi (x,\xi)=\sum_{k=1}^{n_{\varphi}-1} a_kx^{1-\frac{k}{m_{\varphi}}}+\xi x^{1-\frac{n_{\varphi}}{m_{\varphi}}},n_{\varphi}, m_{\varphi}\in \N
,\ \gcd\{k: a_k \neq 0 \}=1,
\label{eq1}
\end{equation}
where $\xi$ is a parameter.  For convenience, we denote $\psi
\prec \varphi$ if $\varphi(x,\xi)=\psi (x, c+{\text{\rm lower
terms in }} x)$. We can fix a coordinate $(x,y)$ such that $P$ and
$Q$ are monic in $y$, i.e $ \deg_yP=\deg P$  and $\deg_yQ=\deg Q$.
For each $\varphi \in \Pi$  we represent
\begin{equation}
\begin{split}
P(x, \varphi (x, \xi))=p_{\varphi} (\xi)x^{\frac{a_{\varphi}}{m_{\varphi}}} +{\text {\rm lower terms in }}
x, 0\neq  p_{\varphi} \in \C[\xi] \\
 Q(x, \varphi (x,
\xi))=q_{\varphi} (\xi)x^{\frac{b_{\varphi}}{m_{\varphi}}} +{\text{\rm
lower terms in }} x, 0\neq q_{\varphi} \in \C[\xi] \\
 J(P,Q)(x,
\varphi (x, \xi))=j_{\varphi} (\xi)x^{\frac{J_{\varphi}}{ m_{\varphi}}}
+{\text{\rm lower terms in }} x,  0\neq j_{\varphi} \in
\C[\xi].\end{split}\label{eq2} 
\end{equation}
 Note that $a_{\varphi}, b_{\varphi}$ and $J_{\varphi} $ are integer numbers.

 A series $\varphi \in \Pi$ is {\it a horizontal
series} of $P$ ( of $Q$ ) if $a_{\varphi} =0$ and $\deg p_{\varphi}
>0$ (resp. $b_{\varphi} =0$ and $\deg q_{\varphi} >0$),  $\varphi $ is
a {\it dicritical series }of $f=(P,Q)$ if $\varphi$ is a
horizontal series of $P$ or $Q$ and $\max
\{a_{\varphi},b_{\varphi}\}=0$  and $\varphi$ is a {\it singular
series} of $f$ if $\deg j_{\varphi} >0$. Note that for every
singular series $\varphi $  of $f$ the equation $J(P,Q)(x,y)=0$
always has a root $y(x)$ of the form $\varphi (x, c+{\text{\rm
lower terms in }} x)$, which gives a branch curve at infinity of
the curve $J(P,Q)=0$. We have the following relations:

i) If $f$ (resp. $P$, $Q$) tends to a finite value along   a
branch curve at infinity $\gamma$, then there is a dicritical
series $\varphi $ of $f$ (resp. a horizontal series $\varphi$ of
$P$, a horizontal series $\varphi$ of $Q$) such that $\gamma $ can
be represented by a Newton-Puiseux of the form $\varphi (x, c+
{\mbox {\rm lower terms in }} x)$;

ii) If $\varphi $ is a dicritical series of $f$ and
$$f(x, \varphi (x, \xi))=f_{\varphi} (\xi) +{\text{\rm lower terms in }} x;$$
then $\  \deg f_{\varphi} >0$ and  its image  is a component of
$A_f$.

iii) (Lemma 4 in [6])$$ A_f = \bigcup_{\varphi  \ is\ a\
dicritical\ series\ of\  f} f_{\varphi} (\C). $$

 This note is to present the following relation between the
singularity set of $f$ and the non-proper value set $A_f$ in terms
of  Newton-Puiseux expansion.

\begin{theorem}\label{theo1}
Suppose $\psi\in \Pi$, $a_\psi> 0$ and $b_\psi>0$,
$(a_\psi,b_\psi)=(Md,Me)$, $M\in \N$, $\gcd(d,e)=1$. Assume that
$\varphi \in \Pi $ is a dicritical series of $f$ such that
$\psi\prec \varphi$. If  $\psi$ is not a singular series of $f$,
then

\begin{enumerate}
\item[(i)] $ (\deg p_\psi, \deg q_\psi)=(Nd,Ne)$ for some $N\in
\N$,

\item[(ii)] $a_{\varphi}=b_{\varphi} =0$ and
\begin{equation}
\begin{split}
p_{\varphi}(\xi)=Lcoeff(p_\psi)C^d\xi^{ D  d}+\dots\\
q_{\varphi} (\xi)=Lcoeff(q_\psi)C^e \xi^{ D  e}+\dots
\end{split}\label{eq3a}
\end{equation}
for some $C\in \C^*$ and  $D\in \N$.

 \end{enumerate}
\end{theorem}

Here, $Lcoeff(h)$ indicates the coefficient of the leading term of
$h(\xi)\in \C[\xi]$.

Theorem \ref{theo1} does not say anything about the existence of
dicritical series $\varphi$, but only shows some properties of
pair $ \psi \prec\varphi$. Such analogous observations  for the
case of non-zero constant Jacobian polynomial map $f$ was obtained
earlier in [7].

For the case when $J(P,Q)\equiv const. \neq 0$, from Theorem
\ref{theo1} (ii) it follows that  if $A_f\neq \emptyset$, then
every irreducible components of $A_f$ can be parameterized by
polynomial maps $\xi\mapsto (p(\xi),q(\xi))$ with
\begin{equation}
\deg p/\deg q=\deg P/\deg Q. \label{eq4} 
\end{equation}
This fact was presented in [6] and can be reduced from [3]. The
estimation (\ref{eq4})  together with the Abhyankar-Moh Theorem on
embedding of the line to the plane in [1] allows us to obtain that
{\it a non-constant polynomial map $f$ of $\C^2$ must have
singularities if its non-proper value set $A_f$ has an irreducible
component isomorphic to the line}. In fact, if $A_f$ has a
component $l$ isomorphic to $\C$, by Abhyankar-Moh Theorem one can
choose a suitable coordinate so that $l$ is the line $v=0$. Then,
every dicritical series $\varphi$ with $f_{\varphi}(\C)=l$ must
satisfy $a_{\varphi}=0$ and $b_{\varphi} <0$. For this situation
we have
\begin{theorem}\label{theo2}
{\it Suppose $\varphi $ is a dicritical
series $\varphi $ of $f$ with $a_{\varphi}=0$ and $b_{\varphi} <0$.
Then, either $\varphi$ is a singular series of $f$ or there is a
horizontal series $\psi$ of $Q$ such that $\psi$ is a singular
series of $f$ and $\psi\prec \varphi$.}\end{theorem}

The proof of Theorem \ref{theo1} presented in the next sections 2-
4 is based on those in [7]. The proof of Theorem \ref{theo2} will
be presented in Section 5.

\section{Associated sequence of pair $\psi\prec
\varphi$.}

From now on, $f=(P,Q):\C^2 \longrightarrow \C^2$ is a given
polynomial map, $P,Q\in \C[x,y]$. The coordinate $(x,y)$ is chosen
so that $P$ and $Q$ are polynomials monic in $y$, i.e.
$\deg_yP=\deg P$ and $\deg_y Q=\deg Q$. Let $\psi , \varphi\in
\Pi$ be given. In this section and the two next sections 3-4 we
always  assume that $\psi $ is not a singular series of $f$,
$\varphi$ is a dicritical series of $f$ and $\psi \prec \varphi$.

Let us represent
\begin{equation}
\varphi (x, \xi)= \psi(x,0)+\sum_{ k=0}^{ K-1} c_kx^{1-\frac{n_k}{m_k} }+\xi x^{ 1-
\frac{ n_K}{ m_K}}, 
\label{eq5}
\end{equation}
where $\frac{n_\psi}{m_\psi}=
\frac{n_0}{m_0} <\frac{n_1}{m_1}< \dots <\frac{n_{K-1}}{m_{K-1}}<
\frac{n_K}{m_K}=\frac{n_{\varphi}}{ m_{\varphi}}  $ and  $c_k\in \C$
may be the zero,  so that  the sequence of series $ \{\varphi_i
\}_{ i=0,1\dots,K}$ defined by
\begin{equation}
\varphi_i(x, \xi ):=\psi(x,0)+\sum_{ k=0}^{ i-1} c_kx^{1-\frac{n_k}{m_k}}+\xi x^{1-\frac{n_i}{m_i}},  i=0, 1,\dots ,  K-
1,\label{eq6} 
\end{equation}
and $\varphi_K:=\varphi$ satisfies the following
properties:

S1) $m_{\varphi_i}=m_i$.

S2) For every $i < K$ at least one of polynomials $p_{\varphi_i}$
and $q_{\varphi_i}$ has a zero point different from the zero.

S3) For every $ \phi (x, \xi)=\varphi_i(x, c_i)+\xi x^{1-\alpha}$,
$ \frac{n_i}{m_i} <\alpha < \frac{n_{i+1}}{m_{i+1}}, $ each of the
polynomials $p_\phi$ and $q_\phi$ is either constant or a monomial
of $\xi. $

The representation (\ref{eq5}) of $\varphi$ is thus the longest
representation such that for each index $i$ there is a
Newton-Puiseux root $y(x)$ of $P=0$ or $Q=0$ such that
$y(x)=\varphi_i(x,c+\mbox{ lower terms in } x)$, $c\neq 0$ if
$c_i=0$. This representation and  the associated sequence $
\varphi_0\prec \varphi_1\prec \dots \prec\varphi_K=\varphi$ is
well defined and unique. Further, $\varphi_0=\psi$.

For simplicity in notations,   below we shall  use lower indices
``$i$" instead of the lower indices ``$\varphi_i$".

For each  associated series $\varphi_i$, $i=0.\dots, K$, let us
represent
\begin{equation}
\begin{split}
P(x, \varphi_i (x, \xi))= p_i(\xi)x^\frac{ a_i}{ m_i}+\text{\rm lower terms in } x \\
Q(x, \varphi_i (x, \xi))=q_i (\xi)x^\frac{ b_i}{ m_i }+\text{\rm lower terms in } x,\\
\end{split}
\label{eq7}
\end{equation}
where $ p_i, q_i \in \C[\xi]-\{0\}$,   $a_i,b_i\in \Z$
and $m_i:=\mult (\varphi_i)$.

\medskip
The property that $P$ and $Q$ are polynomials monic in $y$ ensures
that the Newton-Puiseux roots at infinity $y(x)$ of each equations
$P(x,y)=0$ and $Q(x,y)=0$ are fractional power series of the form
$$y(x)=\sum_{k=0}^\infty c_kx^{1-{k\over m}}, \ m\in \N, \ \gcd\{ k:c_k\neq 0\}=1,$$
for which the map $\tau \mapsto (\tau^m,y(\tau^m))$ is meromorphic
and injective for $\tau$ large enough . Let $\{ u_i(x), i=1,\dots
\deg P\}$ and $\{v_j(x),j=1,\dots \deg Q\}$ be the collections of
the Newton-Puiseux roots of $P=0$ and $Q=0$, respectively.  In
view of the Newton theorem we can represent
\begin{equation}
P(x,y)=A\prod_{i=1}^{\deg P}(y-u_i(x)),\quad
Q(x,y)=B\prod_{j=1}^{\deg Q}(y-v_i(x)). \label{eq8} 
\end{equation}

We refer the
readers to [2] and [5] for the Newton theorem and the
Newton-Puiseux roots.

For each  $i=0.\dots, K$, let us define

- $S_i:=\{k: 1\leq k\leq \deg P: u_k(x)=\varphi_i(x,a_{ik}+\mbox{
lower terms in } x),  a_{ik}\in\C\}$;

- $T_i:=\{k: 1\leq k\leq \deg Q: v_k(x)=\varphi_i(x,b_{ik}+\mbox{
lower terms in } x), b_{ik}\in\C\}$;

- $S_i^0:=\{k \in S_i: a_{ik}=c_i\}$;

 - $T_i^0:=\{k \in T_i: b_{ik}=c_i\}.$\\
Represent

$$p_i(\xi)=A_i\bar p_i(\xi)(\xi-c_i)^{\# S_i^0}, \bar p_i(\xi):=\prod_{k\in S_i\setminus S_i^0}(\xi-a_{ik}),$$
and
$$q_i(\xi)=B_i\bar q_i(\xi)(\xi-c_i)^{\# T_i^0}, \bar q_i(\xi):=\prod_{k\in T_i\setminus T_i^0}(\xi-b_{ik}).$$
Note that $A_i=Lcoeff(p_i)$ and $B_i=Lcoeff(q_i)$.

\begin{lemma}\label{lem2}
{\it For $i=1,\dots , K$
$$A_i=A_{i-1}\bar p_{i-1}(c_{i-1}), \deg p_i=\# S_i=\#S_{i-1}^0,$$
$${a_i\over m_i}={a_{i-1}\over m_{i-1}} + \# S_{i-1}^0(\frac{n_{i-1}}{m_{i-1}}-\frac{n_i}{m_i}),$$

$$B_i=B_{i-1}\bar q_{i-1}(c_{i-1}), \deg q_i=\# T_i=\#T_{i-1}^0,$$
$${b_i\over m_i}={b_{i-1}\over m_{i-1}} + \# T_{i-1}^0(\frac{n_{i-1}}{m_{i-1}}-\frac{n_i}{m_i}).$$}\end{lemma}

\begin{proof}
Note that $\varphi_0(x,\xi)=\psi (x,\xi)$ and
$\varphi_i(x,\xi)=\varphi_{i-1}(x,c_{i-1})+\xi
x^{1-\frac{n_i}{m_i}}$ for $i>0$. Then, substituting
$y=\varphi_i(x,\xi)$, $i=0,1,\dots , K$,  into the Newton
factorizations of $P(x,y)$ and $Q(x,y)$ in (\ref{eq8}) one can
easy verify the conclusions. \end{proof}

\section{Polynomials $j_i(\xi)$}
  Let $\{ \varphi_i\}$
be the associated series of the pair $\psi \prec \varphi$. Denote
$$\Delta_i(\xi):=a_ip_i(\xi)\dot q_i(\xi)-b_i\dot p_i(\xi) q_i(\xi).$$
As assumed, $\psi$ is not a singular series of $f$. So, we have
 $$J(P,Q)(x, \psi (x,\xi))=j_\psi  x^{J_\psi \over m_\psi} +{\text{\rm
lower terms in }} x,  j_\psi\equiv const.  \in \C^* $$ and
$$J(P,Q)(x, \varphi_i (x,\xi))=j_i x^{J_i \over m_i} +{\text{\rm
lower terms in }} x,  j_i\equiv const.  \in \C^* $$ for $
i=0,\dots, K$.

\begin{lemma}\label{lem3}
{\it Let $0\leq i <K$.  If $a_i >0$ and
$b_i>0$, then
$$
\Delta_i(\xi) \equiv
\begin{cases}
-m_i j_i& \text{ if } a_i+b_i=2m_i-n_i+J_i,
\\
0& \text{ if } a_i+b_i>2m_i-n_i+J_i.\end{cases}
$$
Further, $\Delta_i(\xi)\equiv 0$ if and only if $p_i(\xi)$ and
$q_i(\xi)$ have a common zero point. In this case
$$p_i(\xi)^{b_i}=Cq_i(\xi)^{a_i}, \ C\in \C^*.$$
}\end{lemma}

\begin{proof}
Since $a_i >0$ and $ b_i> 0,$  taking
differentiation of $Df(t^{-m_i},\varphi_i (t^{-m_i},\xi ))$,  we
have that
$$m_i j_it^{-J_i+n_i-2m_i-1} +\text{ higher terms in }t
=-\Delta_i (\xi) t^{-a_i-b_i-1}+ \text{ higher terms in }t.
$$
Comparing two sides of it we  can get the first conclusion. The
remains are left to the readers as an elementary exercise.
\end{proof}

\section{Proof of Theorem \ref{theo1}}

 Consider the associated sequence $\{\varphi_i\}_{i=1}^K$ of the pair $\psi
\prec \varphi$.  Since $\varphi$ is a dicritical series of $f$ and
$a_\psi=a_0 >0,b_\psi= b_0 >0$, we can see that $$\deg p_0
>0, \deg q_0>0.$$
Represent $(a_0,b_0)=(Md,Me)$ with $\gcd (d,e)=1$.
 Without loss of
generality we can assume that
$$\deg p_K >0, \ a_K=0 \text{ and } b_K\leq 0.$$
Then, from the construction of the sequence $\varphi_i$ it follows
that
\begin{equation}
\begin{cases}
p_i(c_i)=0 \text{ and } a_i > 0, & \ i=0, 1,\dots ,K-1\\
q_i(c_i)=0& \text{ if }   b_i >0 \end{cases}
\label{eq9}
\end{equation}
Then, by induction using Lemma \ref{lem2}, Lemma \ref{lem3} and
(\ref{eq9}) we can obtain without difficulty the following.

\begin{lemma}\label{lem4}
{\it For $i=0,1,\dots ,K-1$ we have
$$a_i>0,b_i>0,\eqno{\rm (a)}$$
$${a_i\over b_i}={\#S_i\over \# T_i}={d\over e}\eqno{\rm (b)}$$
and$${\# S_i^0\over \#T_i^0}={d\over e}, \bar p_i(\xi)^e=\bar
q_i(\xi)^d . \eqno{\rm (c)}$$}
\end{lemma}

Now, we are ready to complete the proof.

First note that $\deg p_\psi=\# S_0$ and $\deg q_\psi=\#T_0$.
Then, from Lemma \ref{lem4} (c) it follows that
$$ (\deg p_\psi, \deg q_\psi)=(Nd,Ne)$$ for $N=\gcd (\deg p_\psi, \deg q_\psi)\in
\N$. Thus, we get Conclusion (i).

Next, we will show $b_K=0$. Indeed, by Lemma \ref{lem2} (iii) and Lemma \ref{lem4}
(b-c) we have

\begin{eqnarray*}
\frac{b_K}{m_K}&=&\frac{b_{K-1}}{m_{K-1}}+\#T^0_{K-1}(\frac{n_{K-1}}{m_{K-1}}-\frac{n_K}{m_K})\\
&=&\frac{e}{d}[\frac{a_{K-1}}{m_{K-1}}+\#S^0_{K-1}(\frac{n_{K-1}}{m_{K-1}}\frac{n_K}{m_K})]\\
&=&\frac{e}{d}\frac{a_K}{m_K}\\
&=&0,
\end{eqnarray*}

\noindent as $a_K=0$. Thus, we get
$$a_K=b_K=0.$$

Now, we detect the form of polynomials $p_K(\xi)$ and $q_K(\xi))$.
Using Lemma \ref{lem2} (ii-iii) to compute the leading coefficients $A_K$
and $B_K$ we can get
$$A_K=A_0(\prod_{k\leq K-1}\bar p_k(c_k))
, B_K=B_0(\prod_{k\leq K-1}\bar q_k(c_k)).$$ Let $C$ be a
$d-$radical of $(\prod_{k\leq K-1}\bar p_k(c_k))$. Then, by Lemma
\ref{lem2} (ii) and Lemma \ref{lem4} (c) we have that
$$A_K=A_0C^d,B_K=B_0C^e.$$
Let  $D:=\gcd (\# S^0_{K-1},\# T^0_{K-1})$. Then, by Lemma \ref{lem4} (b-c)
we get $$\deg p_K=\# S^0_{K-1}=D d, \deg q_K=\# T^0_{K-1}=D e.$$
Thus,
\begin{align*}
p_K(\xi)=A_0C^d\xi^{ D d}+\dots\\ q_K (\xi)=B_0C^e \xi^{ D
e}+\dots\end{align*} This proves Conclusion (ii).  \hfill
$\square$

\section{Proof of Theorem \ref{theo2}}
Suppose $\varphi $ is a
dicritical series $\varphi $ of $f$ with $a_{\varphi}=0$ and
$b_{\varphi} <0$. Since $b_{\varphi} <0$, there is a horizontal series
$\psi$ of $Q$ such that $\psi\prec \varphi$. We will show that
$\psi$ is a singular series of $f$.

Observe that  $\varphi$ is a horizontal series of $P$ since
$a_{\varphi}=0$. Hence, $\deg p_\psi >0$, since $\psi \prec
\varphi$. Represent
\begin{align*}
P(x,\psi (x, \xi))&=p_\psi (\xi)x^{a_\psi\over m_\psi} +{\text{\rm
lower terms in }} x, \\
 Q(x, \psi (x, \xi))&=q_\psi (\xi)+{\mbox
{\rm lower terms in }} x,  \\
 J(P,Q)(x, \psi (x, \xi))&=j_\psi
(\xi)x^{J_\psi\over m_\psi} +{\text{\rm lower terms in }} x.
\end{align*}
Since $a_\psi >0$ and $ b_\psi=0$, taking differentiation of
$Df(t^{-m_\psi},\psi (t^{-m_\psi},\xi ))$ we have that
$$m_\psi j_\psi(\xi)t^{J_\psi+n_\psi-2m_\psi-1} +\text{ h.terms in }t
=-a_\psi p_\psi(\xi) \dot q_\psi(\xi) t^{-a_\psi-1}+ \text{
h.terms in }t.
$$
Comparing two sides of it we   get that
$$m_\psi j_\psi(\xi)=-a_\psi  p_\psi(\xi) \dot q_\psi(\xi).$$
As $\deg p_\psi >0$, we get $\deg j_\psi(\xi) >0$, i.e. $\psi$ is
a singular series of $f$.\hfill $\square$

\section{Last comment}
To conclude the paper we want to
note that instead of the polynomial maps $f=(P,Q)$ we may consider
 pairs $f=(P,Q)\in k((x))[y]^2$, where $k$ is an algebraically closed
field of zero characteristic and $k((x))$ is the ring of formal
Laurent series in variable $x^{-1}$ with finite positive power
terms. Then, in view of the Newton theorem the polynomial $P(y)$
and $Q(y)$ can be factorized into linear factors in $k((x))[y]$.
And the notions of {\it horizontal series}, {\it dicritical
series} and {\it singular series} can be introduced in an
analogous way. In this situation the statements of Theorem
\ref{theo1} and Theorem \ref{theo2} are still valid and can be
proved in the same way as in sections 2-5.

\section*{Acknowledgments} The author wishes to thank
Prof. A. van den Essen for many valuable suggestions and useful
discussions.

\end{document}